\documentclass[11pt]{amsart}
\usepackage{palatino, mathpazo}
\usepackage[mathscr]{eucal}
\usepackage{amsmath,amsfonts,amssymb, mathrsfs, mathtools, bm}
\usepackage[all]{xy}
\usepackage{hyperref}
\usepackage{paralist}

\usepackage{ytableau}

\newtheorem{theorem}{Theorem}[section]
\newtheorem{lemma}[theorem]{Lemma}
\newtheorem{proposition}[theorem]{Proposition}
\newtheorem{notation}[theorem]{Notation}
\newtheorem{corollary}[theorem]{Corollary}
\newtheorem{remark}[theorem]{Remark}
\newtheorem{example}[theorem]{Example}
\newtheorem{definition}[theorem]{Definition}

\numberwithin{equation}{section}

\numberwithin{equation}{section}

\newcommand{\Img}{\mathrm{Im}}
\newcommand{\Sing}{\mathrm{Sing}}

\newcommand{\Bl}{\mathrm{Bl}}

\newcommand{\rk}{\mathrm{rank}}

\newcommand{\uHom}{\underline{\mathrm{Hom}}}
\newcommand{\IC}{\mathbf{IC}}

\newcommand{\sV}{\mathscr{V}}

\newcommand{\sO}{\mathscr{O}}
\newcommand{\sL}{\mathscr{L}}
\newcommand{\sE}{\mathscr{E}}
\newcommand{\sF}{\mathscr{F}}
\newcommand{\sC}{\mathscr{C}}
\newcommand{\sQ}{\mathscr{Q}}

\newcommand{\cO}{\mathcal{O}}
\newcommand{\cI}{\mathcal{I}}

\newcommand{\bC}{\mathbb{C}}
\newcommand{\bQ}{\mathbb{Q}}

\newcommand{\bP}{\mathbb{P}}
\newcommand{\bD}{\mathbb{D}}
\newcommand{\bN}{\mathbb{N}}

\title{A note on nodal determinantal hypersurfaces}

\author[S.-S.\ Wang]{Sz-Sheng Wang\textsuperscript{\dag}}

\address{\textsuperscript{\dag}Shing-Tung Yau Center and School of Mathematics, Southeast University, Nanjing 211189, China}

\email{103200059@seu.edu.cn}

\begin{document}

\maketitle

\begin{abstract}

We prove that a general determinantal hypersurface of dimension 3 is nodal. Moreover, in terms of Chern classes associated with bundle morphisms, we derive a formula for the intersection homology Euler characteristic of a general determinantal hypersurface.






\end{abstract}

\section{Introduction} \label{intro}

A large and important class of varieties is the class of determinantal varieties. These are a particular case of \emph{degeneracy loci} $D_i (\sigma)$ associated with a morphism $\sigma :\sE \to \sF$ of vector bundles on $M$, which is locally defined by the ideal generated by all $(i + 1)$-minors of a matrix for $\sigma$.

On the other hand, it is well known that a hypersurface 
$$
x_0 g_{\delta}(x_0, \cdots, x_4) + x_1 f_{\delta}(x_0, \cdots, x_4) = 0
$$
in $\bP^4$ is a \emph{nodal} determinantal hypersurface, that is, all its singular points are ordinary double points (ODPs). Here, $g_{\delta}$ and $f_{\delta}$ are general polynomials of degree $\delta$. In this note, we generalize this to a general determinantal hypersurface of dimension $3$. 


To state our result, we need an appropriate notion of the generality of morphisms. Suppose that $M$ is a smooth proper variety and $\sE$, $\sF$ have the same rank $n + 1$. A morphism $\sigma : \sE \to \sF$ is said to be \emph{$n$-general} if for all $0 \leqslant i \leqslant n$, the subset $D_i(\sigma) \setminus D_{i - 1}(\sigma)$ is smooth of codimension $(n + 1 - i)^2$ in the smooth proper variety $M$ (see Definition \ref{general}). 



\begin{theorem} [Theorem \ref{ODPcoro}] \label{introODP}
Suppose that the smooth proper variety $M$ has dimension 4. If a morphism $\sigma $ is $n$-general, then the determinantal hypersurface $D_n(\sigma)$ is nodal, provided that it is irreducible.
\end{theorem}

This gives us a way to construct $3$-dimensional Calabi--Yau (or Fano) hypersurfaces with ODPs (see Example \ref{exampleCY} and \ref{exampleFano}) admitting a natural small resolution (see \eqref{constrsmallreso} and Proposition \ref{chernZ}).


To prove Theorem \ref{introODP}, we will give a formula for the intersection homology Euler characteristic of $D_n (\sigma)$, for a $n$-general morphism $\sigma$. The global topology of the determinantal hypersurface imposes strong restrictions on the singularities $D_n (\sigma)$ can have.

We set $\sL \coloneqq \det \sE^{\vee} \otimes \det \sF$. Note that the maximal degeneracy locus $D_n(\sigma)$ is defined by $\det (\sigma)$ associated to a global section of $\sL$. It is known that
$$
\chi (M | \sL) \coloneqq \int_M c_1(\sL) (1 + c_1(\sL))^{- 1} c(T_M)
$$
is the topological Euler characteristic of the smooth hypersurface defined by a global section of $\sL$ (cf.~\cite[Example 4.2.6]{fulton98}).

Let $d \coloneqq \dim M \geqslant 4$. We let $\chi_{IH} (-)$ denote the intersection homology Euler characteristic. As an analogy of the generalized Milnor number of singular hypersurfaces in \cite{PP95b}, we write 
\begin{equation*} \label{muintro}
    \mu_{IH}(D_n(\sigma), M) := (- 1)^d ( \chi_{IH} (D_{n}(\sigma)) - \chi(M | \sL)).
\end{equation*}
In terms of Schur polynomials $s_{\lambda} (\sF - \sE)$ (see Notation \ref{notationSeg}) we prove:

\begin{theorem}[Theorem \ref{mainresultpf}] \label{mainintro}
If $\sigma : \sE \to \sF$ is $n$-general, then
$$
\mu_{IH}(D_n(\sigma), M) = \sum_{\lambda \supseteq (2,2)} (- 1)^{d + |\lambda|} f^{\lambda} \int_M s_{\lambda} (\sF - \sE) \, c_{d - | \lambda |} (T_M),
$$
where $f^{\lambda}$ denotes the number of standard Young tableaux of shape $\lambda$.
\end{theorem}

Here, we write $\lambda \supseteq \mu$ if the Young diagram of $\mu$ is contained in that of $\lambda$ and the integer $|\lambda|$ for the total number of boxes of the diagram $\lambda$. We adopt the convention that $c_i (T_M) = 0$ for $i < 0$.

In the following cases, $\mu_{IH}(D_n(\sigma), M)$ is expressed in terms of the singular locus of $D_n(\sigma)$, being $D_{n - 1}(\sigma)$ (see Remark \ref{singdeg}), for a $n$-general $\sigma$.

\begin{corollary} [Corollary \ref{maincoro}]
Assume that the smooth proper variety $M$ satisfies one of the following assumptions 
\begin{inparaenum}
    \item \label{3D} $\dim M = 4$, or
    \item \label{4D} $\dim M = 5$ and 
    \begin{equation} \label{cycondiintr}
        c_1(T_M) = c_1(\sF) - c_1(\sE).
    \end{equation}
\end{inparaenum}
If $\sigma$ is $n$-general, then
$$
\mu_{IH}(D_n(\sigma), M)  = (\dim M - 2) \int_{M} c(T_M) \cap [D_{n - 1} (\sigma)].
$$
\end{corollary}

The equality \eqref{cycondiintr} is called the \emph{Calabi--Yau condition}.

The above results are motivated by considering the special case where the bundle $\sF$ is trivial and $\sE$ is a direct sum of line bundles. We proved the case \ref{3D} in this situation \cite[Proposition 3.4]{SSW2016}. It supply a proof of the famous result that all $3$-dimensional Calabi--Yau complete intersections in product of projective spaces are connected through conifold transitions (see \cite{GH88} or \cite[Theorem 5.6]{SSW2016}). When the rank of $\sF$ is two, the case \ref{4D} was proved in \cite{BLS97} for $M$ being a complete intersection in a product of projective spaces.

The paper is organized as follows. In Section \ref{prelsec}, we review some basic materials on Young diagrams, Schur polynomials and intersection cohomology groups. Section \ref{detersec} contains a brief exposition of the required $n$-generality of degeneracy loci. We conclude with Section \ref{appexamsec} where we prove main results and derive interesting formulas (see Proposition \ref{interprod}) for maximal degeneracy loci under the Calabi--Yau condition. 

\medskip

\textbf{Acknowledgements.}
The author is greatly indebted to Prof.~Chin-Lung Wang, Prof.~Baosen Wu and Prof.~Ching-Jui Lai for many useful discussions, and to Dr.~Yinbang Lin for valuable comments. The author would like to thank the anonymous referee for comments and suggestions and Prof.~I. Cheltsov for his help. The author also thanks Tsinghua University, Southeast University, Yau Mathematical Sciences Center and Shing-Tung Yau Center and School of Mathematics for providing support and a stimulating environment.

\section{Preliminaries} \label{prelsec}

Throughout this paper all varieties are irreducible and defined over complex field $\bC$.


\subsection{Young Diagrams} \label{yd}

Let $\lambda = (\lambda_1 \geqslant \lambda_2 \geqslant \cdots \geqslant \lambda_k \geqslant 0)$ be a partition of size $|\lambda| = \sum \lambda_i$, identified with its Young diagram. It consists of $|\lambda|$ boxes arranged in $k$ left adjusted rows of length $\lambda_1, \cdots, \lambda_k$. We abbreviate$(i, \cdots, i)$ ($k$-times) to $(i)^k$ and $(\lambda_1, \lambda_2, \cdots, i + \mu_1, \cdots, i + \mu_k)$ to $\lambda, (i)^k + \mu$. If $\lambda$ and $\mu$ are two Young diagrams, we write $\mu \subseteq \lambda$ if the diagram of $\mu$ is contained in that of $\lambda$, that is, $\mu_i \leqslant \lambda_i$ for all $i$.





A \emph{standard Young tableau} with shape $\lambda$ is a numbering of a Young diagram $\lambda$ with $1, 2, \cdots , |\lambda|$ in which the numbers appear in ascending order within each row or column from left to right and top to bottom. Let $f^{\lambda}$ denote the number of standard Young tableaux of shape $\lambda$. The numbers $f^{\lambda}$ satisfy the following inductive formula 
\begin{equation} \label{induSTY} 
f^{\mu} = \sum_{\substack{|\lambda| = |\mu| - 1 \\ \lambda \subseteq \mu}} f^{\lambda},   
\end{equation}
since giving a standard tableau with $n$ boxes is the same as giving one with $n - 1$ boxes and saying where to put the $n$th box.

For a Young diagram, each box determines a \emph{hook}, which consists of that box and all boxes in its column below the box or in its row to right of the box. The \emph{hook length} of a box is the number of boxes in its hook. Let $h(\lambda)$ be the product of the hook lengths of all boxes in $\lambda$. We have the following closed formula for the number $f^{\lambda}$, the hook length formula \cite[p. 53]{fulton97}.

\begin{proposition} [\cite{fulton97}] \label{hookleng}
If $\lambda$ is a Young diagram with $n$ boxes, then the number $f^{\lambda}$ of standard Young tableaux with shape $\lambda$ is $n !$ divided by the product of the hook lengths of the boxes. In particular, $f^{\lambda} = \binom{n - 1}{k - 1}$ for each hook $\lambda = (k, 1, \cdots, 1)$.
\end{proposition}

Let us illustrate the above definitions and notations with an example.

\begin{example}
$\lambda = (3, 2)$ is drawn as \ytableausetup{smalltableaux}\ydiagram{3, 2}. The number $f^{\lambda}$ of standard Young tableaux is $5$, which has the inductive formula $f^{(3, 2)} = f^{(3, 1)} + f^{(2,2)}$.

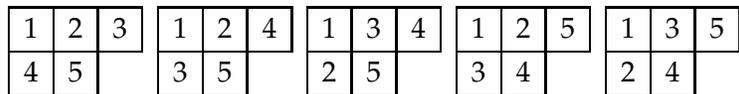
\ytableausetup{nosmalltableaux}
\begin{center}
\ytableausetup{centertableaux}
\begin{ytableau}
1 & 2 & 3 \\
4 & 5 
\end{ytableau}
\ytableausetup{centertableaux}
\begin{ytableau}
1 & 2 & 4 \\
3 & 5 
\end{ytableau}
\ytableausetup{centertableaux}
\begin{ytableau}
1 & 3 & 4 \\
2 & 5 
\end{ytableau}
\ytableausetup{centertableaux}
\begin{ytableau}
1 & 2 & 5 \\
3 & 4 
\end{ytableau}
\ytableausetup{centertableaux}
\begin{ytableau}
1 & 3 & 5 \\
2 & 4 
\end{ytableau}
\end{center}
The product of the hook lengths of all boxes in $\lambda$ is $h(\lambda) = 4 \cdot 3 \cdot 1 \cdot 2 \cdot 1 = 24$. 
\end{example}

\subsection{Schur Polynomials} \label{spoly}

Let $\{c_i\}_{i \in \bN}$ and $\{s_i\}_{i \in \bN}$ be commuting variables and $c_0 = s_0 = 1$, related by the identity $(\sum_{i = 0}^{\infty} c_i t^i) \cdot (\sum_{i = 0}^{\infty} (- 1)^i s_i t^i) = 1$.

\begin{example} \label{chernsym}
For vector bundles $\sE$ and $\sF$, write 
$$
c(\sF - \sE) := c(\sF) / c(\sE) = 1 + (c_1(\sF) - c_1(\sE)) + \cdots
$$
and let $c_i(\sF - \sE)$ be the $i$th term in this expansion. If $c_i = c_i (\sF - \sE)$, then 
$$
s_i = (- 1)^i c_i (\sE - \sF) = c_i (\sE^{\vee} - \sF^{\vee})
$$ 
and thus $s_1 = c_1$. In particular, if $\sE = \sO$, then $s_i = s_i (\sF^{\vee})$ are the Segre classes of the dual bundle $\sF^{\vee}$.
\end{example}

Given a partition $\lambda = (\lambda_1, \cdots, \lambda_k)$, we define the corresponding \emph{Schur polynomial}
\begin{equation*} 
    s_{\lambda} := \det (s_{\lambda_i - i + j})_{1 \leqslant i, j \leqslant k}.
\end{equation*}
Note that adding an arbitrary string of zeros to $\lambda$ does not change $s_{\lambda}$. 

\begin{notation} \label{notationSeg}
We will denote the Schur polynomial by $s_{\lambda}(\sF - \sE)$  associated with $s_i = c_i (\sE^{\vee} - \sF^{\vee})$. To distinguish between these and Segre classes , we will use $s_{(i)}(\sF -\sE)$ to denote $s_i = c_i (\sE^{\vee} - \sF^{\vee})$ for all $i \geqslant 1$.
\end{notation}
 
The following lemma is a consequence of Pieri's formula (cf.~\cite[p. 24]{fulton97} or \cite[Lemma 14.5.2]{fulton98}),
\begin{equation} \label{Pieri}
    s_{1} \cdot s_{\lambda} = \sum\nolimits_{\mu} s_{\mu}
\end{equation}
where the sum is over all partitions $\mu$ such that $\mu \supseteq \lambda$ and $|\mu| = |\lambda| + 1$.

\begin{lemma} \label{simpPieri}
For $\ell \in \bN$, $s_{1}^{\ell} = \sum_{|\lambda| = \ell} f^{\lambda} s_{\lambda}$.
\end{lemma}

\begin{proof}
We argue by induction on $\ell$. From \eqref{Pieri} and the induction hypothesis, we get 
\begin{align*}
    s_{1}^{} \cdot s_{1}^{ \ell - 1} &= \sum_{|\lambda| = \ell - 1} f^{\lambda} s_{1} \cdot s_{\lambda} \\
    &= \sum_{|\lambda| = \ell - 1} f^{\lambda} (\sum_{\substack{|\mu| = \ell \\ \mu \supseteq \lambda}} s_{\mu}) = \sum_{|\mu| = \ell} (\sum_{\substack{|\lambda| = \ell - 1 \\ \mu \supseteq \lambda}} f^{\lambda}) s_{\mu}.
\end{align*}
The lemma now follows immediately from the inductive formula \eqref{induSTY}.
\end{proof}



\subsection{Intersection Cohomology Groups}





Let $V$ be a variety of dimension $m$. Let $\IC^{\bullet}_V$ denote the \emph{intersection cohomology complex} for the (lower) middle perversity on $V$ (see \cite[\S 2.1]{GM83} or \cite[\S 3]{GM83} for Deligne's construction). If $V$ is smooth, then $\IC^{\bullet}_V = \bQ_{V} [2m]$ is just the constant sheaf $\bQ$ put in degree $2m$. For an integer $i \geqslant 0$, the intersection cohomology group $IH^i (V)$ is defined in terms of the hypercohomology groups of $\IC^{\bullet}_V$, that is, $$
IH^i(V) \coloneqq \mathscr{H}^{i -2m} (V, \IC^{\bullet}_V).
$$

\begin{example} \label{intcoh}
Let $V$ be a proper variety of dimension $m$ with isolated singularities and let $\Sigma = \Sing (V)$. For the lower middle perversity, we have the following description of intersection cohomology groups in terms of ordinary ones (cf. \cite[Proposition 4.4.1]{KW06}):
\begin{displaymath}
IH^i (V) = \left\{
\begin{array}{ll}
    H^i(V) & i > m, \\
    \Img (H^i(V) \to H^i(V \setminus \Sigma)) & i = m, \\
    H^i(V \setminus \Sigma) & i < m.
\end{array}
\right.
\end{displaymath}
\end{example}

For an isolated singularity $a \in V$, let us denote by $\mu_0 (V, a)$ the $(m - 1)$th Betti number of the link of the germ of the singularity $(V, a)$. The following proposition is a straightforward generalization of \cite[(6.4.23)]{Dimca92}, where it is stated for hypersurfaces.


\begin{proposition}[\cite{Dimca92}] \label{differU}
Let $m \geqslant 3$ and $V$ be an $m$-dimensional proper variety with only isolated complete intersection singularities. Then 
\begin{equation*}
    \chi_{IH} (V) = \chi (V) + (- 1)^{m + 1} \sum_{a \in \Sing V} \mu_0 (V, a),
\end{equation*}
where $\chi_{IH}(-)$ is the intersection homology Euler characteristic.
\end{proposition}

For the convenience of the reader, we include a proof.


\begin{proof}
Let $\Sigma = \Sing (V)$, and let $h^i(V)$, $h^i_{\Sigma}(V)$ and $Ih^i(V)$ be the dimension of $H^i(V)$, $H^i_{\Sigma}(V)$ and $IH^i(V)$ respectively. According to Example \ref{intcoh} and the following long exact sequence 
\begin{equation*} \label{mhslong}
   \cdots \to H^i_{\Sigma}(V) \to H^i(V) \to H^i(V \setminus \Sigma) \to \cdots \to H^m
    (V) \to IH^m (V) \to 0,
\end{equation*}
it follows that 
\begin{align*}
    \chi_{IH}(V) - \chi(V) &= \sum_{i < m} (- 1)^i (h^i (V \setminus \Sigma) - h^i (V)) + (- 1)^m (Ih^m (V) - h^m (V)) \\
    &= \sum_{i = 0}^m (- 1)^{i + 1} h^i_{\Sigma} (V) = \sum_{i = 0}^m (- 1)^{i + 1} \left(\sum_{a  \in \Sigma} h^{i}_{\{ a \}} (V) \right).
\end{align*}
Since $V$ has only isolated singularities, the $(i - 1)$th reduced cohomology group of the link of the germ $(V, a)$ is isomorphic to $H^i_{\{ a \}} (V)$ (by replacing $V$ with a contractible open neighborhood of $a \in \Sigma$). Then the proposition follows from the fact that the link of $(V, a)$, being a local complete intersection singularity, is $(m - 2)$-connected (see \cite[Corollary 1.3]{Hamm71}).
\end{proof}

\begin{remark} \label{mu0defect}
For an isolated \emph{rational} singularity $(V, a)$, the number $\mu_0 (V, a)$ measures how far $(V, a)$ is from being $\bQ$-factorial. Indeed, we have the \emph{defect} 
\begin{equation*}
    \sigma(V,a) \coloneqq \rk (\mathrm{Weil}(V,a) / \mathrm{Cart}(V,a))
\end{equation*}
of the germ $(V, a)$, where $\mathrm{Weil}(V, a)$ (resp.~$\mathrm{Cart}(V, a)$) is the Abelian group of Weil (resp.~Cartier) divisors of $(V, a)$. It is a finite number \cite[Lemma 1.12]{Kawa88}. If $\sigma (V, a) = 0$, the germ $(V, a)$ is called \emph{$\bQ$-factorial}. On the other hand, the defect $\sigma(V, a)$ equals $h^m_{\{ a \}}(V)$ (cf.~the proof of \cite[Proposition 3.10]{NS95}) and thus equals the $(m - 1)$th Betti number $\mu_0(V, a)$ of the link of $(V, a)$  as seen in the proof of Proposition \ref{differU}.
\end{remark}

We recall the definition of small morphisms in the sense of Goresky and MacPherson \cite{GM83}.

\begin{definition} \label{smallmor}
Let $\pi : W \to V$ be a proper morphism of varieties. For every integer $k \geqslant 1$ we define
$$
V^k \coloneqq \{ x \in V \mid \dim \pi^{- 1}(x) = k \}.
$$
We say that $\pi$ is \emph{small} if $\dim W - \dim V^k > 2 k$ for each $k$.
\end{definition}

\begin{remark} \label{mu0defect3D}
Suppose that $V$ is a $3$-dimensional projective variety with isolated rational singularities that admits a small resolution $\pi : W \to V$. Then the exceptional set of the small resolution $\pi$ consists of irreducible (rational) curves, and the defect $\sigma(V, a)$, as well as $\mu_0 (V, a)$, equals the number of the irreducible components of $\pi^{- 1}(a)$ (see \cite[Lemma 3.4]{Kawa88}). 
\end{remark}

The following proposition connects the topology of $W$ and $V$.

\begin{proposition} [{\cite[\S 6.2]{GM83}}] \label{smallinvchi} \label{invsma}
If a proper morphism $\pi : W \to V$ is a small resolution, then $IH^i(V) = IH^i (W) = H^i(W)$ for all $i$.
\end{proposition}

In particular, the topological Euler characteristic $\chi(W)$ of the smooth variety $W$ equals the intersection homology Euler characteristic $\chi_{IH}(V)$ of $V$.


\subsection{Chern Classes}

We collect some basic material from Fulton's book \cite{fulton98} that we will need in Section \ref{appexamsec}.

\begin{lemma} \label{chernformula}
Let $\sE$ and $\sF$ be vector bundles with the same rank $r$ and $\sL$ a line bundle. Set $\xi = c_1 (\sL)$. Then for each integer $k \geqslant 1$
\begin{align}
    c_k(\sF \otimes \sL - \sE \otimes \sL) &= \sum_{i = 1}^k (- 1)^{k - i} \binom{k - 1}{i - 1} c_i(\sF - \sE) \xi^{k - i}, \label{ch1} \\
    c_r(\sE \otimes \sL) &= \sum_{i \geqslant 0} c_{i}(\sE)  \xi^{r - i}. \label{ch2}
\end{align}
\end{lemma}

\begin{proof}
By \cite[Example 3.2.2]{fulton98}, the Chern polynomial of $\sF \otimes \sL$ is $$
c_{t} (\sF \otimes \sL) = (1 + \xi t)^r c_{\tau(t)}(\sF)
$$
where $\tau(t) = t / (1 + \xi t)$, and similarly for $\sE \otimes \sL$. Then, by substituting one into $t$, we get
\begin{equation} \label{cherndifften}
    c(\sF \otimes \sL - \sE \otimes \sL)  = c_{\tau(1)} (\sF) / c_{\tau(1)} (\sE) = \sum_{i \geqslant 0} c_i(\sF - \sE) (1 + \xi)^{- i}.
\end{equation}
The equation \eqref{ch1} follows from the binomial expansion of $(1 + \xi)^{- i}$ and \eqref{cherndifften}. From the splitting principle, it is easy to get the equation \eqref{ch2} (see \cite[Remark 3.2.3 (b)]{fulton98}).
\end{proof}

Let $\sV$ be a vector bundle of rank $r$ on a smooth proper variety $M$. We say that a subvariety $V$ of $M$ is a \emph{complete intersection} with respect to the vector bundle $\sV$ if it is the zero scheme of a regular section of $\sV$, that is, $V$ has codimension $r$ in $M$ (see \cite[B.3.4]{fulton98}). Then the fundamental class of $V$ is \begin{equation} \label{fumdcl}
    [V] = c_{r} (\sV) \cap [M] 
\end{equation}
(see \cite[Example 3.2.16]{fulton98}) and the restriction $\sV|_V$ is the normal bundle of $V$ in $M$. The \emph{Fulton class} of $V$ is defined by
$$
c^F (V) \coloneqq c(T_M |_V - \sV |_V) \cap [V],
$$
where $T_M$ is the tangent bundle of $M$ (cf.~\cite[Example 4.2.6]{fulton98}). We will denote by $\chi (M | \sV)$ the degree of $0$-dimensional component of the Fulton class $c^F(V)$ and 
$$
\mu_{IH}(V, M) \coloneqq (- 1)^{\dim M} ( \chi_{IH} (V) - \chi(M | \sV)).
$$ 
We will write $\mu_{IH}(V)$ instead of $\mu_{IH}(V, M)$ when no confusion can arise.

\begin{remark}
If we suppose that $V$ is smooth, at least a rational homology manifold, then $\mu_{IH}(V)$ is equal (up to a sign) to the degree of the $0$-dimensional components of the (motivic) Milnor class of $V$ (see \cite{Yokura2010} and references therein).
\end{remark}


The following proposition is probably well known. For the reader’s convenience, we provide a proof of this proposition.

\begin{proposition} \label{interMuIH}
Let $m \geqslant 3$ and $V$ be a $m$-dimensional complete intersection with respect to $\sV$. If $V$ has only isolated singularities, then
\begin{equation} \label{ODPcri}
    \mu_{IH} (V) = (- 1)^{\rk \sV - 1} \sum_{a \in \Sing (V)} \left( \mu_0 (V, a) + \mu (V, a) \right),
\end{equation}
where $\mu(V, a)$ is the Milnor number and $\mu_0 (V, a)$ is the $(m - 1)$th Betti number of the link of the singularity $(V, a)$.
\end{proposition}

\begin{proof}
Note that $m = \dim M - \rk \sV$. By \cite[Theorem 2.4]{SS98}, we have
$$
\chi (V) = \chi(M | \sV) + (- 1)^{m + 1} \sum_{a \in \Sing (V)} \mu (V, a).
$$
The formula \eqref{ODPcri} follows from Proposition \ref{differU} and the definition of $\mu_{IH}(V)$.
\end{proof}


\begin{remark}
For a hypersurface $V$ with (possibly nonisolated) singularities, there is a formula for $\mu_{I H}(V)$ in terms of a stratification of the singular locus (cf.~\cite[Corollary 4.4]{Laurentiu11}).
\end{remark}


\section{Determinantal Contractions} \label{detersec}








Let $\sigma : \sE \rightarrow \sF$ be a morphism of vector bundles of ranks $e$ and $f$ on a smooth proper variety $M$. Note that there is a natural bijection between morphisms $\sE \rightarrow \sF$ and global sections of $\sE^{\vee} \otimes \sF$.

For $i \leqslant \mathrm{min}(e, f)$, we define the $i$th degeneracy locus of $\sigma$ by
$$
D_i(\sigma) = \{ x \in M \mid \rk(\sigma (x)) \leqslant i \}
$$
with the convention $D_{- 1 }(\sigma) = \varnothing$. Its ideal is locally generated by $(i + 1)$-minors of a matrix for $\sigma$. Notice that the $0$th degeneracy locus of $\sigma$ is the zero scheme $Z(\sigma)$. The codimension of $D_i(\sigma)$ in $M$ is less than or equal to $(e - i)(f - i)$ (see Remark \ref{generalPP} below or \cite[Theorem 14.4 (b)]{fulton98}), which is called its \emph{expected codimension}.

\begin{definition}[\cite{PP95a}] \label{general}
For a given integer $r \geqslant 0$, we say that $\sigma$ is $r$-general if for every $i = 0,1, \cdots, r$, the subset $D_i(\sigma) \setminus D_{i - 1}(\sigma)$ is smooth of (expected) codimension $(e - i)(f - i)$ in the smooth variety $M$.
\end{definition}



Before proceeding further, let us remark on the generality of morphisms. 

\begin{remark}\label{generalPP}
Parusi\'nski and Pragacz \cite{PP95a} defined the notion of $r$-general morphisms for an arbitrary variety $M$ with a fixed Whitney stratification $\mathcal{S}$. Indeed, consider the vector bundle $\uHom (\sE, \sF)$ and the universal degeneracy loci $\mathbb{D}_i$ whose fiber over $x \in M$ consists of all homomorphisms $\sE(x) \to \sF(x)$ of rank $\leqslant i$. The codimension of $\bD_i$ in $\uHom (\sE, \sF)$ is $(e - i)(f - i)$ and each $\bD_i \setminus \bD_{i - 1}$ is smooth. If $\sigma : \sE \to \sF$ is viewed as 
$$
s_{\sigma}: M \to \uHom (\sE, \sF)
$$
a section of the vector bundle $\uHom (\sE, \sF)$, then the degeneracy locus $D_i(\sigma)$ is the pullback $s_{\sigma}^{- 1}(\bD_i)$, which is isomorphic to the intersection of $s_{\sigma}(M)$ with $\bD_i$, and thus it has codimension at most $(e - i) (f - i)$ in $M$. Then $\sigma$ is $r$-\emph{general} in the sense of \cite{PP95a} if the section $s_{\sigma}$ induced by $\sigma$ intersects, on each stratum of $\mathcal{S}$, the subset $\bD_i \setminus \bD_{i - 1}$ transversely for $i = 0, 1, \cdots, r$. It coincides with Definition \ref{general} if $M$ is smooth \cite[Lemma 2.9]{PP95a}.
\end{remark}

\begin{remark} \label{singdeg}
According to Definition \ref{general}, it follows that for $i = 0, \cdots, r$, the singular locus of $D_i(\sigma)$ is contained in $D_{i - 1}(\sigma)$ if $\sigma$ is $r$-general. Moreover, the inclusion is an equality. Indeed, since the statement is local, we pick a connection on the vector bundle $\uHom (\sE, \sF)$. Then
$$
\nabla (\wedge^{i + 1} s_{\sigma}) = (i + 1) (\nabla s_{\sigma}) \wedge (\wedge^i s_{\sigma}),  
$$
by the Leibniz rule. Therefore the derivative of $\wedge^{i + 1} s_{\sigma}$ vanishes along $D_i(\sigma)$, that is, $\Sing(D_i(\sigma)) \supseteq D_{i - 1}(\sigma)$. 

If the $(r - 1)$th degeneracy locus has codimension exactly $\dim M$, then the number of points of it is the degree of $[D_{n - 1}(\sigma)]$, because $s_{\sigma}(M)$ and $\bD_{r - 1}$ intersect transversely. 
\end{remark}

From now on we assume that $\sE$ and $\sF$ have the same rank $n + 1$. 

\begin{definition}
For a morphism $\sigma : \sE \to \sF$, we call the maximal degeneracy locus $D_n(\sigma)$ is a \emph{determinantal hypersurface} if it has the expected codimension (which equals one).
\end{definition}


Note the $D_n(\sigma)$ is the zero scheme of the section $\det \sigma$ of the line bundle $\det \sE^{\vee} \otimes \det \sF$.

\begin{example} \label{examp}
The prototypical setting is where $M = \bP^d$ and the vector bundles are sums of line bundles. A morphism of rank $n + 1$ vector bundles
$$
\sigma : \bigoplus \sO_{\bP^d}( - a_i) \to \bigoplus \sO_{\bP^d}(b_j)
$$
is represented as a $(n + 1) \times (n + 1)$ matrix $[\sigma_{i j}(z)]$ of homogeneous polynomials such that $a_i + b_j$ is the degree of $\sigma_{ij}(z)$ for $1 \leqslant i, j \leqslant n + 1$. Then $n$th degeneracy locus $D_n(\sigma)$ is defined by the determinant of the matrix $[\sigma_{i j} (z)]$.
\end{example}

To get small resolutions of maximal degeneracy loci, we consider the following geometric construction (cf.~\cite[Example 14.4.10]{fulton98}). Let $\bP (\sF)$ be the projective bundle with the structure morphism $p$. Fix a morphism $\sigma$ and let $Z_n(\sigma)$ denote the zero scheme in $\bP (\sF)$ of the global section of $p^{\ast} \sE^{\vee} \otimes \sO_{\bP(\sF)} (1)$ induced by the composition of $p^{\ast} \sigma$ with $p^{\ast} \sF \twoheadrightarrow \sO_{\bP(\sF)} (1)$. 



We can interpret $Z_n(\sigma)$ as the projectivization of the cokernel sheaf $\sC$ of $\sigma$. Indeed, thinking of $\bP(\sF)$ as the bundle of $1$-dimensional quotient of $\sF$
$$
\bP(\sF) = \{ (x, \lambda) \mid \lambda : \sF(x) \twoheadrightarrow \bC \},
$$
the zero scheme $Z_n(\sigma)$ induced by 
$p^{\ast} \sE \rightarrow p^{\ast} \sF \twoheadrightarrow \sO_{\bP(\sF)} (1)$ is exactly
$$
\{(x, \lambda) \mid \lambda \circ \sigma (x) : \sE (x) \to \bC \text{ is zero} \},
$$
which is nothing but the subscheme $\bP(\sC)$ of $\bP(\sF)$. So $Z_n(\sigma)$ maps \emph{onto} $D_n(\sigma)$ by $p$, because the linear map $\sigma :\sE(x) \to \sF(x)$ is not surjective if and only if there is a nonzero functional $\lambda : \sF(x) \twoheadrightarrow \bC$ with $\lambda \circ \sigma(x)$ being the zero map.



Let $d = \dim M \geqslant 4$. Denoting by $\pi$ the restriction of $p$ to $Z_n(\sigma)$, we have the following commutative diagram:

\begin{equation} \label{constrsmallreso}
\xymatrix{
Z_n(\sigma) \ar@{^{(}->}[r]^{\jmath} \ar[d]_{\pi} & \bP(\sF)  \ar[d]^p \\
D_n(\sigma) \ar@{^{(}->}[r]^{\iota} & M \mathrlap{\, .}}
\end{equation}
The morphism $\pi$ is called the \emph{determinantal contraction} of $Z_n (\sigma)$.




The following proposition is a reformulation of \cite[Theorem 2.12]{PP95a} for maximal degeneracy loci.

\begin{proposition} [{\cite{PP95a}}] \label{chernZ}
With the above notation, if $\sigma$ is $n$-general then $\pi$ is a small resolution and 
\begin{equation} \label{chiIHD}
\chi_{IH} (D_n(\sigma)) = \chi (Z_n(\sigma)).    
\end{equation}
Moreover, the pushforward of the total Chern class of $Z_n(\sigma)$ is 
\begin{equation} \label{pushfwardch}
    p_{\ast} \jmath_{\ast} c(Z_n (\sigma)) =  \sum_{0 \leqslant i + j \leqslant d - 1} (-1)^{i + j} \binom{i + j}{i} s_{1 + i, (1)^j} (\sF - \sE) \cap c(M).
\end{equation}
\end{proposition}

\begin{proof}
Since $Z_n(\sigma)$ is the projectivization of $\sC = \mathrm{Coker} (\sigma)$, the fiber of $\pi$ over $x \in D_n(\sigma)$ is just the projective space $\bP(\sC (x))$. On the other hand, the projective space $\bP(\sC(x))$ has dimension $\geqslant k$ if and only if $x \in D_{n - k}(\sigma)$. Therefore, if $x \in D_n(\sigma) \setminus D_{n - 1}(\sigma)$ then the fiber $\pi^{- 1}(x)$ is a point. In other word, $\pi$ is birational. 

By the assumption that $\sigma$ is $n$-general, the dimension  of $D_k(\sigma) \setminus D_{k - 1}(\sigma)$ is $d - (n + 1 -k)^2$. It is straightforward to show that the morphism $\pi$ is small (see Definition \ref{smallmor}). Thanks to \cite[Lemma 2.9]{PP95a}, it follows that $Z_n(\sigma)$ is smooth (for details see \cite[pp. 810-811]{PP95a}). Hence the morphism $\pi$ is a small resolution and \eqref{chiIHD} follows from Propositon \ref{smallinvchi}. The formula
\eqref{pushfwardch} of the total Chern class $p_{\ast} \jmath_{\ast} (c(T_{Z_n(\sigma)}) \cap [Z_n(\sigma)])$ is a special case of \cite[Proposition 2.5]{PP95a}.
\end{proof}

\begin{remark}
It is known that the pushforward of the Chern class of a small (or crepant) resolution
is independant of the resolution. The class obtained is called the \emph{stringy Chern class} and the integration of such class is called the \emph{stringy Euler number} \cite{Batyrev97, dLNU07}. In the situation \eqref{constrsmallreso}, the stringy Chern class of $D_n (\sigma)$ and $\pi_{\ast} (c(T_{Z_n(\sigma)}) \cap [Z_n(\sigma)])$ coincide \cite[Proposition 4.5]{dLNU07}.
\end{remark}




\section{Main result} \label{appexamsec}

Keeping the notation introduced in Section \ref{detersec}, we can now prove the formula for the intersection homology Euler characteristic of general determinantal hypersurfaces.

\begin{theorem} \label{mainresultpf}
If the morphism $\sigma : \sE \to \sF$ is $n$-general, then
$$
\mu_{IH}(D_n(\sigma), M) = \sum_{\lambda \supseteq (2,2)} (- 1)^{d + |\lambda|} f^{\lambda} \int_M s_{\lambda} (\sF - \sE) \, c_{d - | \lambda |} (T_M),
$$
where $f^{\lambda}$ denotes the number of standard Young tableaux of shape $\lambda$.
\end{theorem}

\begin{proof}
Let $\sL = \det \sE^{\vee} \otimes \det \sF$. We abbreviate $D_n(\sigma)$ and $Z_n(\sigma)$ to $D$ and $Z$ respectively. Recall that the Fulton class of $D$ is
$$
c^F(D) = c(T_M|_{D} - \sL|_{D}) \cap [D].
$$  
According to Proposition \ref{chernZ} and the commutative diagram \eqref{constrsmallreso}, it follows that
$$
\mu_{IH}(D) = (- 1)^d \int_M ( p_{\ast} \jmath_{\ast} c(Z) - \iota_{\ast} c^F(D)).
$$
Our goal is to find the difference of classes $p_{\ast} \jmath_{\ast} c(Z) - \iota_{\ast} c^F(D)$. First, we observe that $\iota_{\ast} [D] = c_1(\sL) \cap [M]$ in $A_{d - 1}(M)$,
$$
c_1(\sL) = c_1(\sF) - c_1(\sE) = s_{(1)} (\sF - \sE)
$$
and the inverse of total Chern class of $\sL$ is $(1 + c_1(\sL))^{-1}$. By the projection formula,
\begin{align*}
    \iota_{\ast} c^F(D) &= \iota_{\ast} (\iota^{\ast} c(T_M - \sL) \cap [D]) \\
    &= c(T_M - \sL) \cap (c_1(\sL) \cap [M]) \\
    &= (1 + c_1(\sL))^{-1} c_1(\sL) \cap (c(T_M) \cap [M]) \\
    &= \sum_{\ell = 0}^{\infty} (- 1)^{\ell} s_{(1)} (\sF - \sE)^{\ell + 1} \cap c(M).
\end{align*}
Then, by Proposition \ref{chernZ} and the hook length formula (see Proposition \ref{hookleng}),
$$
p_{\ast} \jmath_{\ast} c(Z) =  \sum_{0 \leqslant i + j \leqslant d - 1} (-1)^{i + j} f^{1 + i, (1)^j} s_{1 + i, (1)^j} (\sF - \sE) \cap c(M).
$$
Denote $s_{\lambda}(\sF - \sE)$ briefly by $s_{\lambda}$. According to Lemma \ref{simpPieri}, it follows that for $\ell \geqslant 0$
\begin{align} \label{maincompute}
& \sum_{i + j = \ell } (-1)^{i + j} f^{1 + i, (1)^j} s_{1 + i, (1)^j}^{} - (- 1)^{\ell} s_{(1)}^{\ell + 1} \notag \\
=& (- 1)^{\ell + 1} \left( - \sum_{i + j = \ell } f^{1 + i, (1)^j} s_{1 + i, (1)^j} + \sum_{|\lambda| = \ell + 1} f^{\lambda} s_{\lambda} \right)  \\
=& (- 1)^{\ell + 1} \sum\nolimits_{\lambda \supseteq (2, 2), \, |\lambda| = \ell + 1} f^{\lambda} s_{\lambda} \notag.
\end{align}
The desired formula follows now by \eqref{maincompute} and taking $0$-dimensional components of the difference of classes. This completes the proof.





\end{proof}

The following is an immediate consequence of Theorem \ref{mainresultpf}.


\begin{corollary} \label{maincoro}
Assume that the smooth proper variety $M$ satisfies one of the following assumptions 
\begin{inparaenum}
    \item $\dim M = 4$, or
    \item $\dim M = 5$ and 
    \begin{equation} \label{CYcondi}
        c_1(T_M) = c_1(\sF - \sE).
    \end{equation}
\end{inparaenum}
If $\sigma$ is $n$-general, then 
$$
\mu_{IH}(D_n(\sigma), M)  = (\dim M - 2) \int_{M} c(T_M) \cap [D_{n - 1} (\sigma)].
$$
\end{corollary}

\begin{proof}
For simplicity, we let $s_{\lambda}$ stand for $s_{\lambda} (\sF - \sE)$. For the case $\dim M = 5$, we first observe that the first Chern class of $T_M$ is $s_{(1)}$ by \eqref{CYcondi} (see Example \ref{chernsym}). From Proposition \ref{hookleng}, $f^{(2, 2)} = 2$ and $f^{(3. 2)} = f^{(2, 2, 1)} = 5$. 
Theorem \ref{mainresultpf} now yields  
\begin{align*}
    \mu_{IH}(D_n(\sigma)) &= f^{(3, 2)} \int_M s_{(3, 2)} + f^{(2, 2, 1)} \int_M s_{(2, 2, 1)} - f^{(2, 2)} \int_M s_{(2, 2)} \, c_1(T_M) \\
    &= 5 \int_M s_{(2, 2)} s_{(1)} - 2 \int_M s_{(2, 2)} \, c_1(T_M) = 3 \int_M s_{(2, 2)} \, c_1(T_M).
\end{align*}
Here the second equality follows from Pieri's formula \eqref{Pieri}. Then the corollary follows from the Giambelli-Thom-Porteous formula \cite[Theorem 14.4]{fulton98}, 
\begin{equation} \label{GTPformula}
    [D_{n - 1}(\sigma)] = s_{(2, 2)}(\sF - \sE) \cap [M].
\end{equation}
Similarly, for the case $\dim M = 4$, twice the degree of $[D_{n - 1}(\sigma)]$ is equal to $\mu_{IH}(D_n(\sigma))$.
\end{proof}

\begin{remark}
Corollary \ref{maincoro} also can be computed by standard tools in \cite{fulton98}, e.g., \eqref{jzfundcl} and \eqref{chernTZ} below (cf.~the proof of Proposition \ref{interprod}).
\end{remark}












We are now ready to prove the main result. Recall that if $\sigma : \sE \to \sF$ is $n$-general, then $\pi : Z_n(\sigma) \to D_n(\sigma)$ is a small resolution by Proposition \ref{chernZ}.

\begin{theorem} \label{ODPcoro}
Suppose that the smooth proper variety $M$ has dimension 4. Given a $n$-general morphism $\sigma $, we assume that $Z_n(\sigma)$ is connected. Then the determinantal hypersurface $D_n(\sigma)$ is nodal.
\end{theorem}

\begin{proof}
We abbreviate $D_i(\sigma)$ to $D_i$ for all $i$. By assumption, the determinantal hypersurface $D_n$ is irreducible. Since $\sigma$ is $n$-general and the codimension of $D_{n -  1}$ in $M$ is $4$, the number of singular locus $D_{n -  1}$ of $D_n$ equals the degree of the zero cycle $[D_{n - 1}]$ (see Remark \ref{singdeg}). The theorem will be proved by showing that the Milnor number $\mu (D_{n}, x)$ is $1$ for $x \in D_{n - 1}$.

By Proposition \ref{interMuIH} and Corollary \ref{maincoro},
\begin{equation} \label{ODPcoroeq}
\sum\nolimits_{x \in D_{n - 1}} (\mu_0 (D_{n}, x) + \mu (D_{n}, x)) = \mu_{I H}(D_{n}) = 2 \cdot |D_{n - 1}|. 
\end{equation}
According to that integers $\mu (D_{n}, x)$ and $\mu_0 (D_{n}, x)$ are greater than zero for all $x \in D_{n - 1}$ (see Remark \ref{mu0defect3D}) and \eqref{ODPcoroeq}, it follows that these are equal to one. Hence the $3$-dimensional hypersurface $D_n$ has only ODPs.
\end{proof}







We are going to give a formula of intersections products on the smooth variety $Z_n(\sigma)$. Recall that $\jmath : Z_n(\sigma) \hookrightarrow \bP(\sF)$ is the inclusion, $p : \bP(\sF) \to M$ the bundle map and $d = \dim M$. To simplify notations, we write $Z$ (resp. $D$) instead of $Z_n(\sigma)$ (resp. $D_n(\sigma)$).



\begin{proposition} \label{interprod} 
Fix a $n$-general morphism $\sigma$. Let $\xi = c_1(\sO_{\bP(\sF)} (1))$ and $L = \jmath^{\ast} \xi$. Given a line bundle $\sQ$ on $M$, let $H_M = c_1(\sQ)$ and $H_Z = \jmath^{\ast} p^{\ast} H_M$.  
\begin{enumerate}
    \item \label{internum1} For $0 \leqslant k \leqslant d - 1$,
    $$
    \int_Z H_Z^k . L^{d - 1 -k} = \int_M H_M^k . c_{d - k}(\sE^{\vee} - \sF^{\vee}).
    $$
    \item \label{internum2} Under the assumption $d = 4$ and $c_1(\sF - \sE) = c_1(T_M)$,
    $$
    \int_Z c_2(T_Z).H_Z = \int_M c_2(T_M).c_1(\sE^{\vee} - \sF^{\vee}).H_M,
    $$
    $$
    \int_Z c_2(T_Z).L = \int_M c_2(T_M).c_2(\sE^{\vee} - \sF^{\vee}) - |\Sing (D)|.
    $$
\end{enumerate}
\end{proposition}

\begin{proof}
For simplicity of notation, we write $\sO_{\sF}(1)$ instead of $\sO_{\bP(\sF)}(1)$. Recall that $Z$ is the zero scheme of the global section of $p^{\ast} \sE^{\vee} \otimes \sO_{\sF}(1)$, which is induced by $\sigma : \sE \to \sF$. Then we have, by \eqref{fumdcl} and \eqref{ch2},  
\begin{align} \label{jzfundcl}
     \jmath_{\ast} [Z] &= c_{n + 1} (p^{\ast} \sE^{\vee} \otimes \sO_{\sF}(1)) \cap [\bP(\sF)] \notag \\ 
     &= \sum\nolimits_{i \geqslant 0} c_i(p^{\ast} \sE^{\vee}) . \xi^{n + 1 - i} \cap [\bP(\sF)] \\
     &=  \sum\nolimits_{i \geqslant 0}  \xi^{n + 1 - i} \cap p^{\ast} ( c_i( \sE^{\vee}) \cap [M]) \notag.
\end{align}
Hence, by the projection formula, \eqref{jzfundcl} and the definition of $c_{d - k} (\sE^{\vee} - \sF^{\vee})$,
\begin{align} \label{segrecl}
    p_{\ast} \jmath_{\ast} ( L^{d - 1 -k} \cap [Z]) &= 
    p_{\ast} (\sum\nolimits_{i \geqslant 0} \xi^{n + d - k -i} \cap p^{\ast} (c_i(\sE^{\vee}) \cap [M])) \notag \\
    &= \sum\nolimits_{i \geqslant 0} s_{d - k -i} (\sF^{\vee}).  c_i(\sE^{\vee}) \cap [M]  \\
    &= c_{d - k} (\sE^{\vee} - \sF^{\vee}) \cap [M] \notag,
\end{align}
which establishes the formula \eqref{internum1} of $\int_Z H^{k}_Z . L^{d - 1 - k}$ for $0 \leqslant k \leqslant d -1$.

Now assume that $d = 4$ and $c_1(\sF - \sE) = c_1(T_M)$. We claim that
\begin{equation} \label{chernTZ2}
    c_2(T_Z) = \jmath^{\ast} p^{\ast} ( c_2 (T_M) -  c_2 (\sE^{\vee} - \sF^{\vee}) ) - \jmath^{\ast} p^{\ast} c_1 (\sF^{\vee} - \sE^{\vee}) . L.
\end{equation}
Assume the claim is proved. Then the zero cycle $p_{\ast} \jmath_{\ast} (c_2(T_Z) . L \cap [Z])$ is 
\begin{align*}
    \{[ c_2(T_M) -  c_2(\sE^{\vee} - \sF^{\vee})]. c_2(\sE^{\vee} &- \sF^{\vee}) \\
    &- c_1(\sF^{\vee} - \sE^{\vee}). c_3(\sE^{\vee} - \sF^{\vee})\} \cap [M].
\end{align*}
by \eqref{segrecl}. Since $c_1(\sF^{\vee} - \sE^{\vee}) = - c_1(\sE^{\vee} - \sF^{\vee})$, we get
$$
\int_Z c_2(T_Z).L = \int_M (c_2(T_M). c_2(\sE^{\vee} - \sF^{\vee}) - s_{(2, 2)}(\sE^{\vee} - \sF^{\vee})).
$$
It is easily seen that $s_{(2, 2)}(\sE^{\vee} - \sF^{\vee}) = s_{(2, 2)}(\sF - \sE)$ (cf.~\cite[Lemma 14.5.1]{fulton98}). Then the formula of $\int_Z c_2(T_Z) . L$ follows from the singular locus of $D$ being the $(n - 1)$th degeneracy locus of $\sigma$ and Giambelli-Thom-Porteous formula  \eqref{GTPformula}. The formula of $\int_Z c_2(T_Z) . H_Z$ is proved similarly.

To show the formula of $c_2(T_Z)$, we first recall that by \cite[Example 3.2.11]{fulton98},
\begin{equation*}
    c(T_{\bP(\sF)}) = c(p^{\ast} T_M) \cdot c(p^{\ast} \sF^{\vee} \otimes \sO_{\sF}(1)).
\end{equation*}
From the normal exact sequence
$$
0 \to T_Z \to \jmath^{\ast} T_{\bP(\sF)} \to \jmath^{\ast} (p^{\ast} \sE^{\vee} \otimes \sO_{\sF}(1)) \to 0,
$$
where the right hand term is the normal bundle of $Z$ in $\bP(\sF)$, we have
\begin{align} \label{chernTZ}
    c(T_{Z}) &= \jmath^{\ast} (c(T_{\bP(\sF)}) / c(p^{\ast} \sE^{\vee} \otimes \sO_{\sF}(1))) \\
    &= \jmath^{\ast} (c(p^{\ast} \sF^{\vee} \otimes \sO_{\sF}(1) - p^{\ast} \sE^{\vee} \otimes \sO_{\sF}(1)) \cdot c(p^{\ast} T_M)). \notag
\end{align}
Applying the formula \eqref{ch1} to \eqref{chernTZ}, we get
\begin{align} \label{chernTZorig}
    c_2 (T_Z) = \jmath^{\ast} p^{\ast} ( c_2(T_M) +  c_1 (\sF^{\vee} - \sE^{\vee}). c_1(T_M) +& c_2(\sF^{\vee} - \sE^{\vee})) \notag \\
    &- \jmath^{\ast} p^{\ast} c_1(\sF^{\vee} - \sE^{\vee}) . L. 
\end{align}
According to 
the observation that $c(\sF^{\vee} - \sE^{\vee}) \cdot c(\sE^{\vee} - \sF^{\vee}) = 1$ and $c_1(T_M) = c_1 (\sE^{\vee} - \sF^{\vee})$ by our assumption, it follows that
$$
c_1 (\sF^{\vee} - \sE^{\vee}) \cdot c_1(T_M) + c_2(\sF^{\vee} - \sE^{\vee}) = - c_2 (\sE^{\vee} - \sF^{\vee}),
$$
and \eqref{chernTZ2} is proved.
\end{proof}

\begin{remark}
If we replace \eqref{chernTZ2} with \eqref{chernTZorig}, then the same proof works when we drop the assumption $c_1(T_M) = c_1 (\sE^{\vee} - \sF^{\vee})$. For simplicity, we state Proposition \ref{interprod} \eqref{internum2} under the Calabi--Yau condition. 

\end{remark}

Let us illustrate the above results with examples.

\begin{example} \label{exampleCY}
We use the same notation as in Proposition \ref{interprod} where $M = \bP^4$ and $\sQ = \cO (1)$, and construct quintic hypersurfaces with $46$ ODPs. 

Let $\sF = \cO^{\oplus 4}$ and $\sE = \cO (- 1)^{\oplus 3} \oplus \cO(- 2)$. Given a $3$-general morphism $\sigma : \sE \to \sF$, let $Z = Z_3 (\sigma)$ and $D = D_3 (\sigma)$. Observe that $Z$ is a complete intersection of a $(2, 1)$ and three $(1, 1)$ ample divisors in $\bP (\sF) = \bP^4 \times \bP^3$, and thus it is connected by Lefschetz hyperplane theorem. By Theorem \ref{ODPcoro} and Proposition \ref{interprod}, we obtain Table \ref{exampleCYintsec} (cf.~\cite[Lemma 4.1]{CynkRams15}).

\begin{table}[htbp]
  \caption{Intersection numbers of the quintic containing $B$}\label{exampleCYintsec}
  \centering
  \begin{tabular}{ccccccc}
    \hline
    $L^3$ & $L^2.H$ & $L.H^2$ & $H^3$ & $L.c_2(T_Z)$ & $H.c_2(T_Z)$ & $\#$ of ODPs \\
    \hline
     2 & 7 & 9 & 5 & 44 & 50 & 46 \\
    \hline
  \end{tabular}
\end{table}


This construction includes the example in \cite{CynkRams15}, which is the blow-up of $D$ along a smooth surface $B \subseteq \bP^4$ (cf.~\cite[(4.4)]{CynkRams15}). More precisely, let $B$ be the blow-up of $\bP^2$ at $10$ general points, which is called a Bordiga surface. We can describe $B$ in terms of a degeneracy locus in $\bP^4$. Indeed, the surface $B$ is the locus where the $4 \times 3$ matrix $\tau = [\tau_{ij}]$ of generic linear forms on $\bP^4$ 
has rank $\leqslant 1$. In other words, we have the resolution
\begin{equation*}
    0 \to \cO(- 1)^{\oplus 3} \xrightarrow{\tau} \cO^{\oplus 4} \to \cI_{B}(3) \to 0
\end{equation*}
and $B =D_2 (\tau)$, where $\cI_B$ is the ideal sheaf of $B$ and $\tau$ is $2$-general.



Choose four generic homogeneous polynomials of degree $2$ on $\bP^4$ which induce a morphism $\sigma_1 : \cO (-2) \to \sF$. Then it gives rise to a $3$-general morphism $\sigma = \tau \oplus \sigma_1$ from $\sE$ to $\sF$. Observe that $D = D_3 (\sigma)$ contains the Bordiga surface $B$ and there is a natural morphism $\varphi : Z \to \Bl_B D$ factoring the determinantal contraction $\pi : Z \to D$. Since $\varphi$ is a birational morphism between projective varieties with the same Picard number, the blow-up of $D$ along $B$ is isomorphic to $Z$.  






\end{example}

A series of examples of $3$-dimensional Calabi--Yau varieties will be discussed in detail in a forthcoming paper.






\begin{example} \label{exampleFano}

We also can construct Fano determinantal hypersurfaces in a similar way as Example \ref{exampleCY}. For instance, we consider the case that $M = \bP^4$ and vectore bundles $\sE$, $\sF$ of rank $n + 1$ are direct sums of line bundles with $c_1 (\sF - \sE) = c_1 (\cO (4))$. The determinant of a $n$-general morphism $\sigma :\sE \to \sF$ gives rise to a quartic hypersurface with ODPs in $\bP^4$. 



Abbreviating $\oplus_{i = 1}^r \cO (a_i)$ to $\cO_{\bP^4} (a_1, \cdots, a_r)$, we have the complete list of such $n$-general morphisms $\sigma : \sE \to \sF$ in Table \ref{quarticsinP4} (up to tensor product with a line bundle and dualizing morphisms $\sigma$). The first four cases correspond to \cite[Example 7.3- 7.6]{CHNP13}, which induce building blocks used to construct examples of compact $G_2$-manifold (see \cite{CHNP13} and references therein).


\begin{table}[htbp]
  \centering
  \caption{Quartics in $\bP^4$ contain a special surface} \label{quarticsinP4}
  \begin{tabular}{ccc}
    \hline
    Morphism $\sigma$ & Special surface in $\bP^4$   & $\#$ of ODPs \\ 
       & containing $\Sing(D)$                 \\
    \hline
     $\cO_{\bP^4} (0, 0) \to \cO _{\bP^4}(1, 3)$  & a plane $\Pi$ & 9 \\ 
     $\cO_{\bP^4}(-1, 0) \to \cO_{\bP^4}(1, 2)$ & a quadric surface $Q^2_2$ & 12\\ 
     $\cO_{\bP^4}(0, 0, 0) \to \cO_{\bP^4}(1, 1, 2)$ & a cubic scroll surface $\mathbb{F}$ & 17 \\ 
     $\cO_{\bP^4}(0, 0) \to \cO_{\bP^4}(2, 2)$ & a c.i. of two quadrics $F_{2, 2}$ & 16 \\
     $\cO_{\bP^4}(0, 0, 0, 0) \to \cO_{\bP^4}(1, 1, 1, 1)$ & a Bordiga surface $B$ & 20 \\
    \hline
  \end{tabular}
\end{table}
\end{example}





\begin{example}
We use the same notation with above examples. Let $\sE = \cO_{\bP^4} (0, 0)$ and $\sF = \cO_{\bP^4} (2, 2)$. We pick a $1$-general morpshim $\sigma : \sE \to \sF$ such that $D_1 (\sigma)$ is a nodal quartic hypersurface in $\bP^4$ containing a smooth del Pezzo surface $F_{2, 2}$ of degree $4$ and $|\Sing (D_1 (\sigma))| = 16$. Also, $Z_1 (\sigma) \to D_1 (\sigma)$ is a small resolution of $D_1 (\sigma)$.

Let us apply the formula of $\mu_{IH} (D_1 (\sigma))$ to compute $\chi (Z_1 (\sigma))$. Note that the topological Euler characteristics of a smooth quartic hypersurface is $- 56$. Since $D_1 (\sigma)$ has exactly $16$ ODPs and Corrollary \ref{maincoro}, we get $\mu_{IH} (D_1 (\sigma)) = 32$. By Proposition \ref{chernZ}, $\chi (Z_1 (\sigma)) = \mu_{IH} (D_1 (\sigma)) - 56 = - 24$.

Furthermore, the sufficiently general hypersurface $D_1 (\sigma)$ is nonrational \cite[Theorem 11]{Cheltsov06}. Indeed, I. Cheltsov (see \cite[\S 3]{Cheltsov06}) observed that the restriction of $Z_1 (\sigma)$ to the second projection $\bP (\sF) \to \bP^1$ is a standard del Pezzo fibration of degree $4$. Then $D_1 (\sigma)$ is nonrational if $\chi (Z_1 (\sigma)) \neq 0, -4$ or $-8$.
\end{example}

\begin{remark}
We mention the rationality problem for other (sufficiently general) quartic threefolds $D$ in Table \ref{quarticsinP4}. If $D$ contains $\Pi$ or $Q_2^2$, then it is nonrational (see \cite[Theorem 6, Example 10]{Cheltsov06}, \cite[Corollary 1.11]{CG17}). On the other hand, $D$ is rational if it contains $\mathbb{F}$ or $B$ (see \cite[Entry 30 in Table 1]{Kaloghiros12}, \cite[Example 7]{Cheltsov06}). 
\end{remark}


\begin{thebibliography}{99}
\providecommand{\url}[1]{{#1}}
\providecommand{\urlprefix}{URL }
\expandafter\ifx\csname urlstyle\endcsname\relax
  \providecommand{\doi}[1]{DOI~\discretionary{}{}{}#1}\else
  \providecommand{\doi}{DOI~\discretionary{}{}{}\begingroup
  \urlstyle{rm}\Url}\fi

\bibitem{Batyrev97}
Batyrev, V.V.: Stringy {H}odge numbers of varieties with {G}orenstein canonical
  singularities.
\newblock In: Integrable systems and algebraic geometry ({K}obe/{K}yoto, 1997),
  pp. 1--32. World Sci. Publ., River Edge, NJ (1998)

\bibitem{BLS97}
Brunner, I., Lynker, M., Schimmrigk, R.: Unification of {M}- and {F}-theory
  {C}alabi-{Y}au fourfold vacua.
\newblock Nuclear Phys. B \textbf{498}(1-2), 156--174 (1997).
\newblock \urlprefix\url{https://doi.org/10.1016/S0550-3213(97)89481-3}

\bibitem{Cheltsov06}
Cheltsov, I.: Nonrational nodal quartic threefolds.
\newblock Pacific J. Math. \textbf{226}(1), 65--81 (2006).
\newblock \doi{10.2140/pjm.2006.226.65}.
\newblock \urlprefix\url{https://doi.org/10.2140/pjm.2006.226.65}

\bibitem{CG17}
Cheltsov, I., Grinenko, M.: Birational rigidity is not an open property.
\newblock Bull. Korean Math. Soc. \textbf{54}(5), 1485--1526 (2017)

\bibitem{CHNP13}
Corti, A., Haskins, M., {Nordstr\"om}, J., {Pacini}, T.: Asymptotically
  cylindrical {C}alabi-{Y}au 3-folds from weak {F}ano 3-folds.
\newblock Geom. Topol. \textbf{17}(4), 1955--2059 (2013).
\newblock \doi{10.2140/gt.2013.17.1955}.
\newblock \urlprefix\url{https://doi.org/10.2140/gt.2013.17.1955}

\bibitem{CynkRams15}
Cynk, S., Rams, S.: On {C}alabi-{Y}au threefolds associated to a web of
  quadrics.
\newblock Forum Math. \textbf{27}(2), 699--734 (2015).
\newblock \doi{10.1515/forum-2012-0056}.
\newblock \urlprefix\url{https://doi.org/10.1515/forum-2012-0056}

\bibitem{Dimca92}
Dimca, A.: Singularities and topology of hypersurfaces.
\newblock Universitext. Springer-Verlag, New York (1992).
\newblock \urlprefix\url{https://doi.org/10.1007/978-1-4612-4404-2}

\bibitem{dLNU07}
de~Fernex, T., Lupercio, E., Nevins, T., Uribe, B.: Stringy {C}hern classes of
  singular varieties.
\newblock Adv. Math. \textbf{208}(2), 597--621 (2007).
\newblock \doi{10.1016/j.aim.2006.03.005}.
\newblock \urlprefix\url{https://doi.org/10.1016/j.aim.2006.03.005}

\bibitem{fulton97}
Fulton, W.: Young tableaux, \emph{London Mathematical Society Student Texts},
  vol.~35.
\newblock Cambridge University Press, Cambridge (1997).
\newblock With applications to representation theory and geometry

\bibitem{fulton98}
Fulton, W.: Intersection theory, \emph{Ergebnisse der Mathematik und ihrer
  Grenzgebiete. 3.}, vol.~2, second edn.
\newblock Springer-Verlag, Berlin (1998).
\newblock \urlprefix\url{https://doi.org/10.1007/978-1-4612-1700-8}

\bibitem{GM83}
Goresky, M., MacPherson, R.: Intersection homology. {II}.
\newblock Invent. Math. \textbf{72}(1), 77--129 (1983).
\newblock \urlprefix\url{https://doi.org/10.1007/BF01389130}

\bibitem{GH88}
Green, P.S., H\"ubsch, T.: Connecting moduli spaces of {C}alabi-{Y}au
  threefolds.
\newblock Comm. Math. Phys. \textbf{119}(3), 431--441 (1988).
\newblock \urlprefix\url{http://projecteuclid.org/euclid.cmp/1104162497}

\bibitem{Hamm71}
Hamm, H.: Lokale topologische {E}igenschaften komplexer {R}\"aume.
\newblock Math. Ann. \textbf{191}, 235--252 (1971).
\newblock \urlprefix\url{https://doi.org/10.1007/BF01578709}

\bibitem{Kaloghiros12}
Kaloghiros, A.: A classification of terminal quartic 3-folds and applications
  to rationality questions.
\newblock Math. Ann. \textbf{354}(1), 263--296 (2012).
\newblock \doi{10.1007/s00208-011-0658-z}.
\newblock \urlprefix\url{https://doi.org/10.1007/s00208-011-0658-z}

\bibitem{Kawa88}
Kawamata, Y.: Crepant blowing-up of {$3$}-dimensional canonical singularities
  and its application to degenerations of surfaces.
\newblock Ann. of Math. (2) \textbf{127}(1), 93--163 (1988).
\newblock \urlprefix\url{https://doi.org/10.2307/1971417}

\bibitem{KW06}
Kirwan, F., Woolf, J.: An introduction to intersection homology theory, second
  edn.
\newblock Chapman \& Hall/CRC, Boca Raton, FL (2006).
\newblock \urlprefix\url{https://doi.org/10.1201/b15885}

\bibitem{Laurentiu11}
Laurentiu, M.: On {M}ilnor classes of complex hypersurfaces.
\newblock In: Topology of stratified spaces, \emph{Math. Sci. Res. Inst.
  Publ.}, vol.~58, pp. 161--175. Cambridge Univ. Press, Cambridge (2011)

\bibitem{NS95}
Namikawa, Y., Steenbrink, J.H.M.: Global smoothing of {C}alabi-{Y}au
  threefolds.
\newblock Invent. Math. \textbf{122}(2), 403--419 (1995).
\newblock \urlprefix\url{https://doi.org/10.1007/BF01231450}

\bibitem{PP95a}
Parusi\'nski, A., Pragacz, P.: Chern-{S}chwartz-{M}ac{P}herson classes and the
  {E}uler characteristic of degeneracy loci and special divisors.
\newblock J. Amer. Math. Soc. \textbf{8}(4), 793--817 (1995).
\newblock \urlprefix\url{https://doi.org/10.2307/2152829}

\bibitem{PP95b}
Parusi\'nski, A., Pragacz, P.: A formula for the {E}uler characteristic of
  singular hypersurfaces.
\newblock J. Algebraic Geom. \textbf{4}(2), 337--351 (1995)

\bibitem{SS98}
Seade, J., Suwa, T.: An adjunction formula for local complete intersections.
\newblock Internat. J. Math. \textbf{9}(6), 759--768 (1998).
\newblock \urlprefix\url{https://doi.org/10.1142/S0129167X98000324}

\bibitem{SSW2016}
Sz-Sheng, W.: On the connectedness of the standard web of {C}alabi--{Y}au
  3-folds and small transitions.
\newblock Asian J. Math. \textbf{22}(6), 981--1004 (2018).
\newblock \doi{10.4310/AJM.2018.v22.n6.a1}.
\newblock \urlprefix\url{http://dx.doi.org/10.4310/AJM.2018.v22.n6.a1}

\bibitem{Yokura2010}
Yokura, S.: Motivic {M}ilnor classes.
\newblock J. Singul. \textbf{1}, 39--59 (2010).
\newblock \doi{10.5427/jsing.2010.1c}.
\newblock \urlprefix\url{https://doi.org/10.5427/jsing.2010.1c}

\end{thebibliography}


\end{document}